\documentclass[11pt,amsfonts]{article}
\usepackage[margin=1in]{geometry}  
\usepackage{graphicx}              

\usepackage{amsmath, amssymb, amsthm, epsfig}               
\usepackage{enumerate}
\usepackage{amsfonts}              
\usepackage{amsthm}                
\usepackage{amsthm}
\usepackage{enumerate}
\theoremstyle{plain}
\usepackage{color}
\newtheorem{corollary}{Corollary}[section]
\newtheorem{definition}[corollary]{Definition}

\newtheorem{lemma}[corollary]{Lemma}

\newtheorem{thm}[corollary]{Theorem}

\newfont{\sBlackboard}{msbm10 scaled 900}

\newcommand{\mylabel}[1]{\label{#1}
	\ifx\undefined\stillediting
	\else \fbox{$#1$}\fi }
\newcommand{\BE}{\begin{equation}}

\newcommand{\EEQ}{\end{equation}}
\newcommand{\rfb}[1]{\mbox{\rm
		(\ref{#1})}\ifx\undefined\stillediting\else:\fbox{$#1$}\fi}

\newfont{\Blackboard}{msbm10 scaled 1200}

\newfont{\roma}{cmr10 scaled 1200}

\newcommand{\bb}{\begin{equation}}
\newcommand{\bbb}{\end{equation}}

\DeclareMathOperator{\meas}{meas}

\newcommand{\mm}    {{\hbox{\hskip 0.5pt}}}

\newcommand{\bluff} {{\hbox{\raise 15pt \hbox{\mm}}}}
scaled\magstep2

\makeatletter
\def\section{\@startsection {section}{1}{\z@}{-3.5ex plus -1ex minus
		-.2ex}{2.3ex plus .2ex}{\large\bf}}
\makeatother

\begin{document}
\title{Infinitely many solutions for a class of fractional Orlicz-Sobolev Schr\"{o}dinger equations }
\author{Sabri Bahrouni and Hichem Ounaies}

\maketitle

\begin{abstract}
In the present paper, we deal with a new compact embedding theorem for a subspace of the new fractional Orlicz-Sobolev spaces. We also establish some useful inequalities which yields to apply the variational methods. Using these abstract results, we study the existence of infinitely many nontrivial solutions for a class of fractional Orlicz-Sobolev Schr\"{o}dinger equations whose simplest prototype is
$$(-\triangle)^{s}_{m}u+V(x)m(u)u=f(x,u),\ x\in\mathbb{R}^{N},$$
where $0<s<1$, $N\geq2$, $(-\triangle)^{s}_{m}$ is fractional $M$-Laplace operator and the nonlinearity $f$ is sublinear as $|u|\rightarrow\infty$. The proof is based on the variant Fountain theorem established by Zou.
\end{abstract}
{\small \textbf{Keywords:} Fractional Orlicz-Sobolev space, Compact embedding theorem, Fractional $M-$Laplacian, Fountain Theorem.

\section{Introduction and main result}

In this paper, we are concerned with the study of the nonlinear fractional $M$-Laplacian equation:
\begin{equation}\label{eq}
(-\triangle)^{s}_{m}u+V(x)m(u)u= f(x,u),\ x\in\mathbb{R}^{N}, \\
\end{equation}
where $0<s<1$, $N\geq2$ and $M(t)=\displaystyle\int_{0}^{|t|}m(s)sds.$  \\

In the last years, problem \eqref{eq} has received a special attention for the
case where $M(t)=\frac{1}{2}|t|^{2}$, that is, when it is of the form
\begin{equation}\label{eq1}
  -\triangle u+V(x)u= f(x,u),\ x\in\mathbb{R}^{N}.
\end{equation}
We do not intend to review the huge bibliography of equations like \eqref{eq1}, we just emphasize that the most famous conditions on the potential $V:\mathbb{R}^{N}\rightarrow\mathbb{R}$
are the following:\\

\noindent $(V_{1})$ $\quad\ \ \ V\in C(\mathbb{R}^{N},\mathbb{R})$ and $\inf_{x\in\mathbb{R}^{N}}V(x)\geq V_{0}>0$.\\

\noindent $(V_{2})$ \quad\ \ \  There exists $\nu> 0$ such that $$\lim_{|y|\rightarrow+\infty}meas(\{x\in\mathbb{R}^{N}:\ |x-y|\leq \nu,\ V(x)\leq L\})=0,\ \forall L>0,$$
where $meas(.)$ denotes the Lebesgue measure in $\mathbb{R}^{N}$. We quote here \cite{Anouar, Zhang, ZXu} where the existence of infinitely many nontrivial solutions for the equation \eqref{eq1} have been obtained in connection with the geometry of the function $V$.\\

For the case where $s=1$, problem \eqref{eq} becomes
$$-\triangle_{m}u+V(x)m(u)u= f(x,u),\ x\in\mathbb{R}^{N},$$ where the operator $\triangle_{m}u=div(m(|\nabla u|)\nabla u)$ named $M$-Laplacian. The reader can find more details involving this subject in \cite{Alves, Rad, Rad1, Rad2} and their references.\\

Notice that when $0<s<1$ and $M(t)=\frac{1}{p}|t|^{p}$ where $p>1$ the problem \eqref{eq} gives back the fractional Schr\"{o}dinger equation
\begin{equation}\label{eq2}
  (-\triangle)^{s}_{p}u+V(x)|u|^{p-2}u= f(x,u),\ x\in\mathbb{R}^{N} ,
\end{equation}
where $ (-\triangle)^{s}_{p}$ is the non-local fractional $p$-Laplacian operator. Concerning the equation \eqref{eq2}, in the last decade, many several existence and multiplicity results have been obtained by using different variational methods.
In \cite{Rad 2}, the authors studied the existence of multiple ground state solutions for the problem \eqref{eq2}, when the nonlinear term $f$ is assumed to have a superlinear behaviour at the origin and a sublinear decay at infinity.
Ambrosio \cite{Ambrosio} established an existence of infinitly solutions for the problem \eqref{eq2}, when $f$ is $p$-superlinear and $V (x)$ can change sign. Moreover, fractional Schr\"{o}dinger-type problems have been considered in some interesting papers \cite{pucci, chang, Torres}. The literature on non-local operators and on their applications is very interesting and, up to now, quite large. After the seminal papers by Caffarelli {\it et al.} \cite{11,12,13}, a large amount of papers were written on problems involving the fractional diffusion operator $(-\Delta)^{s}$  ($0<s<1$). We can quote \cite{4,14,23,25,26} and the references therein.  We also refer to the recent monographs \cite{14,22} for a thorough variational approach of  non-local problems.\\

Contrary to the classical fractional Laplacian case that is widely investigated, the situation seems to be in a developing state when the new fractional $M$-Laplacian is present. In this context, the natural setting for studying problem \eqref{eq} are fractional Orlicz-Sobolev spaces. Currently, as far as we know, the only results for fractional Orlicz-Sobolev spaces and fractional $M$-Laplacian operator are obtained in \cite{Azroul, Sabri, 7}. In particular, in \cite{7}, the authors define the fractional order Orlicz-Sobolev space associated to an $N$-function $M$ and a fractional parameter $0<s<1$ as
$$W^{s,M}(\Omega)=\bigg{\{}u\in L^{M}(\Omega):\ \int_{\mathbb{R}^{N}}\int_{\mathbb{R}^{N}}M\bigg{(}\frac{u(x)-u(y)}{|x-y|^{s}}\bigg{)}\frac{dxdy}{|x-y|^{N}}<\infty\bigg{\}}.$$

The previous definition creates problems in the calculus and in the embedding results, for example, the Borel measure defined as $d\mu=\frac{dxdy}{|x-y|^{N}}$ is not finish in the neighbourhood of the origin, that's whay, in \cite{Azroul}, the authors introduced another definition of the fractional Orlicz-Sobolev space, i.e,
$$W^{s,M}(\Omega)=\bigg{\{}u\in L^{M}(\Omega):\ \exists\lambda>0/\ \int_{\Omega}\int_{\Omega}M\bigg{(}\frac{\lambda(u(x)-u(y))}{|x-y|^{s}M^{-1}(|x-y|^{N})}\bigg{)}dxdy<\infty\bigg{\}}.$$
The authors in \cite{Sabri} gave some further basic properties both on this function space and the related nonlocal operator. \\

Motivated by the above papers, under the suitable conditions $(V _{1} )$ and $(V_{2} )$ on the potential $V$ and
exploiting the variant Fountain theorem,  we aim to study the multiplicity of nontrivial weak solutions to \eqref{eq} where the new fractional $M$-Laplacian is present. In this spirt, we deal with a new compact embedding theorem, also, we establish some useful inequalities  which yields to apply the variational methods.
As far as we know, all these results are new.\\

 Related to functions $M$ and $f$, our hypotheses are the following:\\

\textbf{Conditions on} $m$ and $M$:\\

The function $m: \mathbb{R}_{+}\rightarrow\mathbb{R}_{+}$ is a $C^{1}$-function satisfying
\begin{enumerate}
  \item [($m_{1}$)] $m(t)$, $(m(t)t)^{'}>0$ for all $t>0$.
  \item [($m_{2}$)] There exist $l,r\in]1,N[$ such that $$l\leq r<l^{*}=\frac{Nl}{N-l}\ \ \ \text{and}\ \ \  l\leq \frac{m(t)t^{2}}{M(t)}\leq r,\ \forall t\neq0,$$
  where $$M(t)=\int_{0}^{|t|}m(s)sds.$$
\end{enumerate}
Moreover, $m_{*}(t)t$ is such that the Sobolev conjugate function $M_{*}$ of $M$ is its primitive; that is, $$M_{*}(t)=\int_{0}^{|t|}m_{*}(s)sds.$$
  $(M_{1})$ \ \ There exists a positive constant $C$ such that
   $$C |t|^{\mu}\leq M_{*}(at),\ \forall a,t\geq 0,\ \forall 1<\mu\leq r,$$
  where $M_{*}$ is the Sobolev conjugate of $M$.\\

\noindent $(M_{2})$ $\lim_{|t|\rightarrow+\infty}\displaystyle\frac{C|t|^{\mu}}{M(t)}=0,\ \forall 1<\mu\leq r.$\\

\noindent $(M_{3})$ The function  $t\mapsto M(\sqrt{t}),\ t\in[0,\infty[$ is convex.\\

 \textbf{Conditions on} $f$:\\

 \noindent $(f_{1})$ $f(x,u)=p\xi(x)|u|^{p-2}u$, where $1<p<l$ is a constant and $\xi:\ \mathbb{R}^{N}\rightarrow\mathbb{R}$ is a positive continuous function such that $\xi\in L^{\frac{r}{r-p}}(\mathbb{R}^{N})$.\\

We mention some examples of functions $M$, whose function $m(t)$ satisfies the conditions $(m_{1})$-$(m_{2})$. The examples are the following:
\begin{enumerate}
  \item $M(t)=|t|^{p}$ for $1<p<N$.
  \item $M(t)=|t|^{p}+|t|^{q}$ for $1<p<q<N$ and $q\in]p,p^{*}[$ with $p^{*}=\frac{Np}{N-p}$.
  \item $M(t)=(1+|t|^{2})^{\gamma}-1$ for $1<\gamma<\frac{N}{N-2}$.
\end{enumerate}


Using the above hypotheses, we are able to state our main result.

\begin{thm}\label{thm1}
  Suppose that $(m_{1})-(m_{2})$, $(M_{1})-(M_{3})$, $(V_{1})-(V_{2})$ and $(f_{1})$ hold. Then, problem \eqref{eq} possesses infinitely many nontrivial solutions.
\end{thm}

This paper is organized as follows. In Section $2$, we give some
definitions and fundamental properties of the spaces
$L^{M}(\Omega)$ and $W^{s,M}(\Omega)$. In Section $3$,  we prove some basic properties of the fractional Orlicz-Sobolev space and we show a compact embedding type theorem. Finally, in Section $4$, using a variant Fountain theorem, we prove our main result.

\section{Preliminaries}

In this preliminary section, for the reader's convenience, we make a brief overview on the fractional Orlicz-Sobolev spaces studied in \cite{Azroul}, and the associated fractional $M$-laplacian operator.\\

 Let $M: \mathbb{R}\rightarrow\mathbb{R}_{+}$ be an $N$-function, i.e,

\begin{enumerate}
	\item $M$ is even, continuous, convex, with $M(t)>t$ for $t>0$,
	\item $\frac{M(t)}{t}\rightarrow 0$ as $t\rightarrow0$ and  $\frac{M(t)}{t}\rightarrow +\infty$ as $t\rightarrow+\infty$.
\end{enumerate}

Equivalently, $M$ admits the representation: $$M(t)=\int_{0}^{|t|}m(s)ds,$$ where $m: \mathbb{R}_{+}\rightarrow\mathbb{R}_{+}$ is non-decreasing, right continuous, with $m(0)=0$, $m(t)>0\ \forall t>0$ and $m(t)\rightarrow\infty$ as $t\rightarrow\infty$. The conjugate $N$-function of $M$ is defined by
$$\overline{M}(t)=\int_{0}^{|t|}\overline{m}(s)ds,$$ where  $\overline{m}: \mathbb{R}_{+}\rightarrow\mathbb{R}_{+}$ is given by
$\overline{m}(t)=\sup\{s:\ m(s)\leq t\}$.
Evidently we have

\begin{equation}\label{Young}
st\leq M(s)+\overline{M}(t),
\end{equation}

which is known as the Young inequality. Equality holds in \eqref{Young} if and only if either $t=m(s)$ or $s=\overline{m}(t)$.\\

In what follows, we say that an $N$-function $M$ verifies the $\triangle_{2}$ condition
\begin{equation}\label{delta2}
  M(2t)\leq K\ M(t),\ \forall t\geq0,
\end{equation}
for some constant $K>0$. This condition can be rewritten in the following way: For each $s>0$, there exists $K_{s}>0$ such that
\begin{equation}\label{deltas}
  M(st)\leq K_{s}\ M(t),\ \forall t\geq0.
\end{equation}

If $A$ and $B$ are two $N$-functions, we say that $A$ is stronger than $B$ if $$B(x)\leq A(ax),\ x\geq x_{0}\geq 0,$$ for each $a>0$ and $x_{0}$ (depending on $a$), $B\prec\prec A$ in symbols. This is the case if and only if for every positive constante $k$ $$\lim_{t\rightarrow+\infty}\frac{B(kt)}{A(t)}=0.$$
The Orlicz class $K^{M}(\Omega)$ (resp. the Orlicz space $L^{M}(\Omega)$) is defined as the set of (equivalence classes of) real-valued measurable
functions $u$ on $\Omega$ such that $$\rho(u;M)=\int_{\Omega}M(u(x))dx<\infty\ (\text{resp.}\ \int_{\Omega}M(\lambda u(x))dx<\infty\ \text{for some}\ \lambda>0).$$
$L^{M}(\Omega)$ is a Banach space under the Luxemburg norm
\begin{equation}\label{19}
\|u\|_{(M)}=\inf\bigg{\{}\lambda>0\ :\ \int_{\Omega}M(\frac{u}{\lambda})\leq1\bigg{\}},
\end{equation}
 whose norm is equivalent to the Orlicz norm $$\|u\|_{L^{M}(\Omega)}=\sup_{\rho(v;\overline{M})\leq1}\int_{\Omega}|u(x)||v(x)|dx.$$

The next lemma  and their proof can be found in \cite{Fukagai}.

 \begin{lemma}\label{lem2}
   Assume that $(m_{1})$ and $(m_{2})$ hold and let $\xi_{0}(t)=\min\{t^{l},t^{r}\}$, $\xi_{1}(t)=\max\{t^{l},t^{r}\}$, for all $t\geq0$. Then,
   $$\xi_{0}(\rho)M(t)\leq M(\rho t)\leq \xi_{1}(\rho)M(t)\ \text{for}\ \rho,t\geq0$$ and $$\xi_{0}(\|u\|_{(M)})\leq \int_{\mathbb{R}^{N}}M(|u|)dx\leq\xi_{1}(\|u\|_{(M)})\ \text{for}\ u\in L^{M}(\mathbb{R}^{N}).$$
 \end{lemma}

\begin{definition}
	Let $M$ be an $N$-function. For a given domain $\Omega$ in $\mathbb{R}^{N}$ and $0<s<1$, we define the fractional Orlicz-Sobolev space $W^{s,M}(\Omega)$ as follows,
	\begin{equation}\label{20}
	W^{s,M}(\Omega)=\bigg{\{}u\in L^{M}(\Omega):\ \exists\lambda>0/\ \int_{\Omega}\int_{\Omega} M\bigg{(}\frac{\lambda(u(x)-u(y))}{|x-y|^{s}M^{-1}(|x-y|^{N})}\bigg{)}dxdy<\infty\bigg{\}}.
	\end{equation}
	This space is equipped with the norm,
	\begin{equation}\label{21}
	\|u\|_{(s,M)}=\|u\|_{(M)}+[u]_{(s,M)},
	\end{equation}
	where $[.]_{(s,M)}$ is the Gagliardo semi-norm, defined by
	\begin{equation}\label{22}
	[u]_{(s,M)}=\inf\bigg{\{}\lambda>0:\ \int_{\Omega}\int_{\Omega} M\bigg{(}\frac{u(x)-u(y)}{\lambda|x-y|^{s}M^{-1}(|x-y|^{N})}\bigg{)}dxdy\leq1\bigg{\}}.
	\end{equation}
\end{definition}

Let $W^{s,M}_{0}(\Omega)$ denote the closure of $C_{c}^{\infty}(\Omega)$ in the norm $\|.\|_{(s,M)}$ defined in \eqref{21}.

\begin{thm}\label{thm2}[(Generalized Poincar\'{e} inequality)]
	Let $\Omega$ be a bounded open subset of $\mathbb{R}^{N}$ and let $s\in]0,1[$. Let $M$ be an $N$-function. Then there exists a positive
	constant $\mu$ such that, $$\|u\|_{(M)}\leq \mu [u]_{(s,M)},\ \ \forall\ u\in W^{s,M}_{0}(\Omega).$$
\end{thm}

Therefore, if $\Omega$ is bounded and $M$ be an $N$-function, then $[u]_{(s,M)}$ is a norm of $ W^{s,M}_{0}(\Omega)$ equivalent to $\|u\|_{(s,M)}$.\\

Let $M$ be a given $N$-function, satisfying the following conditions:

\begin{equation}\label{4}
\int_{0}^{1}\frac{M^{-1}(\tau)}{\tau^{\frac{N+s}{N}}}d\tau<\infty
\end{equation}
and
\begin{equation}\label{5}
\int_{1}^{+\infty}\frac{M^{-1}(\tau)}{\tau^{\frac{N+s}{N}}}d\tau=\infty.
\end{equation}

If \eqref{5} is satisfied, we define the inverse Sobolev conjugate $N$-function of $M$ as follows,

\begin{equation}\label{6}
M_{*}^{-1}(t)=\int_{0}^{t}\frac{M^{-1}(\tau)}{\tau^{\frac{N+s}{N}}}d\tau.
\end{equation}

\begin{thm}\label{inj}
	Let $M$ be an $N$-function and $s\in]0,1[$. Let $\Omega$ be a bounded open subset of $\mathbb{R}^{N}$ with $C^{0,1}$-regularity and bounded boundary. If \eqref{4} and \eqref{5} hold, then
	\begin{equation}\label{7}
	W^{s,M}(\Omega)\hookrightarrow L^{M_{*}}(\Omega).
	\end{equation}
	Moreover,
	\begin{equation}\label{8}
	W^{s,M}(\Omega)\hookrightarrow L^{B}(\Omega)
	\end{equation}
	is compact for all $B\prec\prec M_{*}$.
\end{thm}

The fractional $M$-Laplacian operator is defined as
\begin{equation}\label{17}
(-\triangle)^{s}_{m}u(x)=2 P.V\int_{\mathbb{R}^{N}} m\bigg{(}\frac{u(x)-u(y)}{|x-y|^{s}M^{-1}(|x-y|^{N})}\bigg{)}\frac{u(x)-u(y)}{|u(x)-u(y)|}\frac{dy}{|x-y|^{s}M^{-1}(|x-y|^{N})},
\end{equation}
where $P.V$ is the principal value.\\
This operator is well defined between $W^{s,M}(\mathbb{R}^{N})$ and its dual space  $W^{-s,\overline{M}}(\mathbb{R}^{N})$. In fact, in [\cite{Azroul}, lemma 3.5] the following representation formula is provided
\begin{equation}\label{18}
\langle(-\triangle)^{s}_{m}u,v\rangle=\int_{\mathbb{R}^{N}}\int_{\mathbb{R}^{N}} m\bigg{(}\frac{u(x)-u(y)}{|x-y|^{s}M^{-1}(|x-y|^{N})}\bigg{)}\frac{u(x)-u(y)}{|u(x)-u(y)|}\frac{v(x)-v(y)}{|x-y|^{s}M^{-1}(|x-y|^{N})}dxdy,
\end{equation}
for all $v\in W^{s,M}(\mathbb{R}^{N})$.

\section{Variational setting and some useful tools}

In this section, we will first introduce the variational setting for problem \eqref{eq}. In view of the presence of potential $V(x)$, our working space is $$E=\bigg{\{}u\in W^{s,M}(\mathbb{R}^{N});\ \int_{\mathbb{R}^{N}}V(x)M(u)dx<\infty\bigg{\}},$$
equipped with the following norm $$\|u\|= [u]_{(s,M)}+\|u\|_{(V,M)}$$ where $$\|u\|_{(V,M)}=\inf\bigg{\{}\lambda>0;\ \int_{\mathbb{R}^{N}}V(x)M\bigg{(}\frac{u}{\lambda}\bigg{)}dx\leq1\bigg{\}}.$$
We define the functional $G: E\rightarrow \mathbb{R}$ by
\begin{equation}\label{F}
G(u)=\int_{\mathbb{R}^{N}}\int_{\mathbb{R}^{N}}M(h_{u}(x,y))dxdy,
\end{equation}
where $h_{u}(x,y)=\displaystyle\frac{u(x)-u(y)}{|x-y|^{s}M^{-1}(|x-y|^{N})}.$\\

After integrating, we obtain from $(f_{1})$ that for any $(x,t)\in \mathbb{R}^{N}\times\mathbb{R}$
\begin{equation}\label{F}
  F(x,t)=\int_{0}^{t}f(x,s)ds=\xi(x)|t|^{p}.
\end{equation}

 In order to prove Theorem \ref{thm1}, we will consider the following family of functionals
$$I_{\lambda}(u)=G(u)+\Psi(u)-\lambda B(u)$$
with $\lambda\in[1,2]$, $u\in E$ and
 $$\Psi(u)=\int_{\mathbb{R}^{N}}V(x)M(u)dx,\ \ \ \ B(u)=\int_{\mathbb{R}^{N}}F(x,u)dx.$$
We will show that $I_{\lambda}$ satisfies the assumptions of the following variant of fountain Theorem due to Zou \cite{Zou}.

\begin{thm}\label{Fountain}
  Let $(E,\|\|)$ be a Banach space and $E=\overline{\bigoplus_{j\in\mathbb{N}}X_{j}}$ with $dim\ X_{j}<\infty$ for any $j\in\mathbb{N}$. Set $Y_{k}=\bigoplus_{j=1}^{k}X_{j}$ and $Z_{k}=\bigoplus_{j=k}^{\infty}X_{j}$. Consider the following $C^{1}$-functional $I_{\lambda}:\ E\rightarrow\mathbb{R}$ defined by $$I_{\lambda}(u)=A(u)-\lambda B(u),\ \lambda\in[1,2].$$
  Assume that $I_{\lambda}$ satisfies the following assumptions:

  $(i)$ $I_{\lambda}$ maps bounded sets to bounded sets for $\lambda\in[1,2]$ and $I_{\lambda}(-u)=I_{\lambda}(u)$ for all $(\lambda,u)\in [1,2]\times E$.

  $(ii)$ $B(u)\geq0$, $B(u)\rightarrow+\infty$ as $\|u\|\rightarrow+\infty$ on any finite dimensional subspace of $E$.

  $(iii)$ There exists $r_{k}>\rho_{k}$ such that
  $$a_{k}(\lambda):=\inf_{u\in Z_{k},\|u\|=\rho_{k}}I_{\lambda}(u)\geq 0>b_{k}(\lambda):=\max_{u\in Y_{k},\|u\|=r_{k}}I_{\lambda}(u),\ \forall\lambda\in[1,2],$$ and $$d_{k}(\lambda):=\inf_{u\in Z_{k},\|u\|\leq\rho_{k}}I_{\lambda}(u)\rightarrow0\ \text{as}\ k\rightarrow\infty\ \text{uniformly for}\ \lambda\in[1,2].$$

  Then there exist $\lambda_{n}\rightarrow1$, $u_{\lambda_{n}}\in Y_{n}$ such that
  $$I_{\lambda_{n}}^{'}|_{Y_{n}}(u_{\lambda_{n}})=0,\ I_{\lambda_{n}}(u_{\lambda_{n}})\rightarrow c_{k}\in [d_{k}(2),b_{k}(1)]\ \text{as}\ n\rightarrow\infty.$$
  Particularly, if $(u_{\lambda_{n}})$ has a convergent subsequence for every $k$, then $I_{1}$ has infinitely many nontrivial critical points $\{u_{k}\}\in E\backslash\{0\}$ satisfying $I_{1}(u_{k})\rightarrow0^{-}$ as $k\rightarrow\infty$.
\end{thm}

Now we give the definition of weak solution for the problem \eqref{eq}.
 We define the functional $I_{1}$ on $E$ by $$I_{1}(u)=A(u)-B(u),$$ where
  \begin{align*}
  A(u)&=G(u)+\Psi(u).
  \end{align*}
\begin{definition}
  We say that $u\in E$ is a weak solution to $\eqref{eq}$ if $u$ satisfies $$\langle I_{1}^{'}(u),v\rangle=\langle A^{'}(u),v\rangle-\langle B^{'}(u),v\rangle$$ for all $v\in E$.
\end{definition}

  The functional $I_{\lambda}$ is well defined on $E$ moreover $I_{\lambda}\in C^{1}(E,\mathbb{R})$ and

\begin{equation}\label{I'}
  \langle I_{\lambda}^{'}(u),v\rangle=\langle A^{'}(u),v\rangle-\lambda\langle B^{'}(u),v\rangle\ \forall v\in E.
\end{equation}

Then the critical points of $I_{1}$ are weak solutions to \eqref{eq}.\\

Now, we introduce some important inequalities that show that the functional $I_{\lambda}$ satisfies the hypothesis of Theorem \ref{Fountain}.

\begin{lemma}\label{lem3} we assume that $(m_{1})$, $(m_{2})$ and $(V_{1})$ are satisfied. Then,
  the following properties hold true:\\

    \noindent$(i)$ $\xi_{0}([u]_{(s,M)})\leq G(u)\leq\xi_{1}([u]_{(s,M)})\ \forall u\in E,$\\

    \noindent$(ii)$ $\xi_{0}(\|u\|_{(V,M)})\leq \displaystyle\int_{\mathbb{R}^{N}}V(x)M(u)dx\leq \xi_{1}(\|u\|_{(V,M)})\ \forall u\in E,$
\end{lemma}
\begin{proof}
  $(i)$ By Lemma \ref{lem2}, we Know that $$\xi_{0}(\|u\|_{(M)})\leq \int_{\mathbb{R}^{N}}M(u)dx \leq\xi_{1}(\|u\|_{(M)})\ \forall u\in L^{M}(\mathbb{R}^{N}).$$ It follows that $$\xi_{0}(\|h_{u}\|_{(M,\mathbb{R}^{2N})})\leq G(u)\leq\xi_{1}(\|u\|_{(M,\mathbb{R}^{2N})})\ \forall u\in L^{M}(\mathbb{R}^{N}).$$
  Having in mind that, $\|h_{u}\|_{L^{M}(\mathbb{R}^{2N})}=[u]_{(s,M)},$ we obtain
   $$\xi_{0}([u]_{(s,M)})\leq G(u)\leq\xi_{1}([u]_{(s,M)})\ \forall u\in E,$$
   $(ii)$ Using Lemma \ref{lem2} and  Choosing $\rho=\|u\|_{(V,M)}$, we have
   $$
    M( u)\leq \xi_{1}(\|u\|_{(V,M)})M\bigg{(}\frac{u}{\|u\|_{(V,M)}}\bigg{)},$$
   then, $$
   V(t)M( u)\leq \xi_{1}(\|u\|_{(V,M)})V(t)M\bigg{(}\frac{u}{\|u\|_{(V,M)}}\bigg{)}\ \text{for}\ t\in\mathbb{R}^{N}.$$
   From the definition of the norm \eqref{19}, we obtain,
   $$\int_{\mathbb{R}^{N}}V(t)M( u)dt\leq \xi_{1}(\|u\|_{(V,M)})\int_{\mathbb{R}^{N}}V(t)M\bigg{(}\frac{u}{\|u\|_{(V,M)}}\bigg{)}dt\leq \xi_{1}(\|u\|_{(V,M)}).$$
   Using the similar reasoning with $\rho=\|u\|_{(V,M)}-\epsilon$ and $\epsilon>0$, we get
   $$\xi_{0}(\|u\|_{(V,M)}-\epsilon)V(t)M\bigg{(}\frac{u}{\|u\|_{(V,M)}-\epsilon}\bigg{)}\leq V(t)M( u),$$
   then $$\int_{\mathbb{R}^{N}}V(t)M( u)dt\geq \xi_{0}(\|u\|_{(V,M)}-\epsilon)\int_{\mathbb{R}^{N}}V(t)M\bigg{(}\frac{u}{\|u\|_{(V,M)}-\epsilon}\bigg{)}dt\geq\xi_{0}(\|u\|_{(V,M)}-\epsilon).$$
   Letting $\epsilon\rightarrow0$ in the above inequality, we obtain $$\xi_{0}(\|u\|_{(V,M)})\leq\int_{\mathbb{R}^{N}}V(t)M( u)dt.$$
  The proof of Lemma \ref{lem3} is complete.
\end{proof}
Now we show that the following compactness result holds.
\begin{lemma}\label{lem1}
 We suppose that $(m_{1})$ and $(m_{2})$ are satisfied. Let $\Phi$ be an $N$-function satisfying the $\triangle_{2}$ condition, $\Phi\prec\prec M_{*}$ and
 \begin{equation}\label{Mphi}
 \lim_{|t|\rightarrow+\infty}\frac{\Phi(t)}{M(t)}=0.
 \end{equation}
 Under the assumption $(V_{1})$ and $(V_{2})$, the embedding from $E$ into $L^{\Phi}(\mathbb{R}^{N})$ is compact.
\end{lemma}
\begin{proof}
  Let $(u_{n})\subset E$ be a sequence verifying $u_{n}\rightharpoonup 0\ \text{in}\ E.$ We have to show that $u_{n}\rightarrow0$ in $L^{\Phi}(\mathbb{R}^{N})$. By using Theorem \ref{thm1} we know that $u_{n}\rightarrow0$ in $L^{\Phi}_{loc}(\mathbb{R}^{N})$.
  Thus it suffices to show that, for any $\epsilon>0$, there exists $R > 0$ such that $$\int_{B^{c}_{R}(0)}\Phi(u_{n})dx<\epsilon;$$ here $B^{c}_{R}(0)=\mathbb{R}^{N}\setminus B_{R}(0)$.\\
   Choose $(y_{i})_{i\in\mathbb{N}}\subset\mathbb{R}^{N}$ such that $\mathbb{R}^{N}\subset \bigcup_{i\in\mathbb{N}}B_{\nu}(y_{i})$ and each point $x\in \mathbb{R}^{N}$ is contained in at most $2^{N}$ such balls $B_{\nu}(y_{i})$. Let $$A_{R,L}=\{x\in B_{R}^{c}:\ V(x)\geq L\}\ \ \text{and}\ \ B_{R,L}=\{x\in B_{R}^{c}:\ V(x)< L\}.$$
  The fact that $(u_{n})$ converges weakly to $u$ in $E$ implies that $\|u_{n}\|_{E}\leq T,\ \forall n\in\mathbb{N}$ with $T>0$. From the $\triangle_{2}$ condition, there is $K>0$ such that \begin{equation}\label{K}
                                          M(T \frac{u_{n}}{T})\leq K M(\frac{u_{n}}{T}),\ \forall n\in\mathbb{N}.
                                      \end{equation}
  Given $\epsilon>0$, by \eqref{Mphi}, there is $R>0$ such that              \begin{equation}\label{KK}
                                                             \Phi(t)\leq \epsilon M(t),\ \text{if}\ |t|>R.
                                                           \end{equation}
   Combining \eqref{K} and \eqref{KK}, we get
  \begin{align}\label{ARL}
    \int_{A_{R,L}}\Phi(u_{n})dx &\leq \frac{\epsilon}{L}\int_{\mathbb{R}^{N}}V(x)M(u_{n})dx\\
                                &\leq \frac{\epsilon K}{L}\int_{\mathbb{R}^{N}}V(x)M\bigg{(}\frac{u_{n}}{T}\bigg{)}dx\nonumber\\
                                &\leq\frac{\epsilon K}{L}\int_{\mathbb{R}^{N}}V(x)M\bigg{(}\frac{u_{n}}{\|u_{n}\|_{(V,M)}}\bigg{)}dx\nonumber\\
                                &\leq\frac{\epsilon K}{L}\nonumber
  \end{align}
  and this can be made arbitrarily small by choosing $L$ large.\\

   Take $C$ an $N$-function such that $C\circ\Phi\prec\prec M$ and let $\overline{C}$ be the conjugate of $C$.
  By Theorem $4.17.4$ in \cite{Kufner}, there exist $K^{'}>0$ such that
                         \begin{equation}\label{C}
                           \|u\|_{C\circ\Phi}\leq K^{'} \|u\|_{(M)},\ \forall u\in L^{M}(\mathbb{R}^{N}).
                         \end{equation}

   \textbf{Claim1:} $\xi_{1}(\|u_{n}\|_{(M)})\leq \xi_{1}\bigg{(}\displaystyle\frac{1}{V_{0}}\xi_{1}(\|u_{n}\|_{E})+1\bigg{)}$\\
   Indeed, using Lemma \ref{lem3}, we get \begin{align*}
             \xi_{1}(\|u_{n}\|_{M}) &\leq \xi_{1}\bigg{(}\int_{\mathbb{R}^{N}}M(u_{n})dx+1\bigg{)}\\
                                        &\leq \xi_{1}\bigg{(}\frac{1}{V_{0}}\int_{\mathbb{R}^{N}}V(x)M(u_{n})dx+1\bigg{)}\\
                                        &\leq \xi_{1}\bigg{(}\frac{1}{V_{0}}\xi_{1}(\|u_{n}\|_{(V,M)})+1\bigg{)}\\
                                        &\leq \xi_{1}\bigg{(}\displaystyle\frac{1}{V_{0}}\xi_{1}(\|u_{n}\|_{E})+1\bigg{)}.
           \end{align*}

   We fixe $L>0$. Combining \eqref{C}, claim $1$ and Lemma \ref{lem2} and applying the H\"{o}lder inequality, we infer that
  \begin{align}\label{BRL}
     \int_{B_{R,L}}\Phi(u_{n})dx &\leq \sum_{i\in\mathbb{N}}\int_{B_{R,L}\bigcap B_{\nu}(y_{i})}\Phi(u_{n})dx\\
                                 &\leq \sum_{i\in\mathbb{N}}\|\Phi(u_{n})\|_{L^{C}(B_{\nu}(y_{i}))}\|\chi_{B_{R,L}\bigcap B_{\nu}(y_{i})}\|_{L^{\overline{C}}(B_{\nu}(y_{i}))}\nonumber\\
                                 &\leq \epsilon_{R} \sum_{i\in\mathbb{N}}\|\Phi(u_{n})\|_{L^{C}(B_{\nu}(y_{i}))}\leq \epsilon_{R} 2^{N}\|\Phi(u_{n})\|_{L^{C}(\mathbb{R}^{N})}\nonumber\\
                                 &\leq \epsilon_{R} 2^{N} \bigg{\{}\int_{\mathbb{R}^{N}}C(\Phi(u_{n}))dx+1\bigg{\}}\nonumber\\
                                 &\leq \epsilon_{R} 2^{N}\{ \xi_{1}[\|u_{n}\|_{(C\circ\Phi)}]+1\}\nonumber\\
                                 &\leq \epsilon_{R} 2^{N}\{K^{''}\xi_{1}[\|u_{n}\|_{(M)}]+1\}\nonumber\\
                                 &\leq \epsilon_{R} 2^{N}\bigg{\{}K^{''}\xi_{1}\bigg{[}\frac{1}{V_{0}}\xi_{1}(T)+1\bigg{]}+1\bigg{\}}\nonumber
  \end{align}
  where $\epsilon_{R}=\sup_{y_{i}}\|\chi_{B_{R,L}\bigcap B_{\nu}(y_{i})}\|_{L^{\overline{C}}(B_{\nu}(y_{i}))}$ and $K^{''}>0$. By assumption $(V_{2})$ and Proposition 4.6.9 in \cite{Kufner} we can infer that $\epsilon_{R}\rightarrow0$ as $R\rightarrow\infty$. Thus we may make this term small by choosing $R$ large. Combining \eqref{ARL} and \eqref{BRL} we get our desired result.
\end{proof}

\begin{corollary}\label{cor1}
  Under $(M_{1})$ and $(M_{2})$, the embedding from $E$ into $L^{\mu}(\mathbb{R}^{N})$ is compact for all $1<\mu\leq r$.
\end{corollary}

\begin{proof}
Let $\Phi(t)=C|t|^{\mu}$. By condition $(M_{1})$, $(M_{2})$ and applying Lemma \ref{lem1}, we can deduce that $E$ is compactly embedded in $L^{\mu}(\mathbb{R}^{N})$ for all $1<\mu\leq r$.
\end{proof}

\begin{lemma}\label{lem5}
  The functional $A$ is weakly lower semi-continuous.
\end{lemma}

\begin{proof}
  By Lemma $3.3$ in \cite{Sabri}, it is enough to show that $\Psi$ is weakly lower semi-continuous. Let $(u_{n})\subset E$ be a sequence which converges weakly to $u$ in $E$. Since $E$ is compactly embedded in $L^{r}(\mathbb{R}^{N})$ it follows that $(u_{n})$ converges strongly to $u$ in $L^{r}(\mathbb{R}^{N})$. Then, up to a subsequence, we obtain $$u_{n}(x)\rightarrow u(x),\ a.e\ \text{in}\ \mathbb{R}^{N}.$$
  This along with Fatou's lemma yield $$\Psi(u)\leq \liminf_{n\rightarrow\infty}\Psi(u_{n}).$$
  Therefore, $A$ is weakly lower semi-continuous. The proof of Lemma \ref{lem5} is complete.
\end{proof}

\begin{lemma}\label{lem4}
  If $u_{n}\rightharpoonup u$ in $E$ and
  \begin{equation}\label{PsiG}
  \langle A^{'}(u_{n}),u_{n}-u\rangle\rightarrow0,
  \end{equation}
   then $u_{n}\rightarrow u$ in $E$.
\end{lemma}
\begin{proof}
	Since $(u_{n})$ converges weakly to $u$ in $E$ implies that $([u_{n}]_{(s,M)})$ and $(\|u_{n}\|_{(V,M)})$ are a bounded sequences of real numbers. That fact and relations $(i)$ and $(ii)$ from lemma \ref{lem3} imply that the sequences $(G(u_{n}))$ and $(\Psi(u_{n}))$ are bounded, it means that the sequence $(A(u_{n}))$ is bounded. Then, up to a subsequence,
we deduce that $A(u_{n})\rightarrow c$.
	Furthermore, Lemma \ref{lem5}, implies
	\begin{equation}\label{23}
	A(u)\leq \displaystyle\liminf_{n\rightarrow\infty}
	A(u_{n})=c.
	\end{equation}
	On the other hand, since $A$ is convex, we have
                                                     \begin{equation}\label{24}
                                                       A(u)\geq A(u_{n})+\langle A^{'}(u_{n}),u-u_{n}\rangle.
                                                     \end{equation}
	Therefore, combinings \eqref{23} and \eqref{24} and the hypothesis \eqref{PsiG}, we conclude that $A(u)=c$.\\
	Taking into account that $\displaystyle\frac{u_{n}+u}{2}$ converges weakly to $u$ in $E$ and using again the weak lower semi-continuity of $A$ we find
	\begin{equation}\label{11}
	c=A(u)\leq \displaystyle\liminf_{n\rightarrow\infty}A\bigg{(}\frac{u_{n}+u}{2}\bigg{)}.
	\end{equation}
	We assume by contradiction that $(u_{n})$ does not converge to $u$ in $E$. Then by $(i)$ and $(ii)$ in lemma \ref{lem3} it follows that there exist $\epsilon>0$ and a subsequence $(u_{n_{m}})$ of $(u_{n})$ such that
	\begin{equation}\label{12}
	A\bigg{(}\frac{u_{n_{m}}-u}{2}\bigg{)}\geq\epsilon,\ \forall\ m\in\mathbb{N}.
	\end{equation}
	On the other hand, relations \eqref{delta2} and $(M_{3})$ enable us to apply [\cite{17}, theorem 2.1] in order to obtain
	\begin{equation}\label{13}
	\frac{1}{2}A(u)+\frac{1}{2}A(u_{n_{m}})-A\bigg{(}\frac{u_{n_{m}}+u}{2}\bigg{)}\geq A\bigg{(}\frac{u_{n_{m}}-u}{2}\bigg{)}\geq\epsilon,\ \forall m\in\mathbb{N}.
	\end{equation}
	Letting $m\rightarrow\infty$ in the above inequality we obtain
	\begin{equation}\label{14}
	c-\epsilon\geq\displaystyle \limsup_{m\rightarrow\infty}A\bigg{(}\frac{u_{n_{m}}+u}{2}\bigg{)}.
	\end{equation}
	and that is a contradiction with \eqref{11}. It follows that  $(u_{n})$ converges strongly to $u$ in $E$ and lemma \ref{lem4} is proved.
\end{proof}

\section{Proof of Theorem \ref{thm1}}

We further need the following lemmas.

\begin{lemma}\label{Bpositive}
  Let $(V_{1})$, $(V_{2})$, $(M_{1})$, $(M_{2})$ and $(f_{1})$ be satisfied. Then $B(u)\geq0$. Furthermore, $B(u)\rightarrow\infty$ as $\|u\|\rightarrow\infty$ on any finite dimensional subspace of $E$.
\end{lemma}

\begin{proof}
Evidently $B(u)\geq0$ follows by $(f_{1})$.
  We claim that for any finite dimensional subspace $F \subset E$  there exists a constant $ \epsilon> 0$ such that
  \begin{equation}\label{epsilon}
  \meas(\{x\in\mathbb{R}^{N}:\ \xi(x)|u(x)|^{p}\geq\epsilon\|u\|^{p}\})\geq\epsilon,\ \forall u\in F\setminus{0}.
  \end{equation}
We argue by contradiction and we suppose that for any $n\in\mathbb{N}$ there exists $0\neq u_{n}\in F$ such that
$$\meas(\{x\in\mathbb{R}^{N}:\ \xi(x)|u_{n}(x)|^{p}\geq\frac{1}{n}\|u_{n}\|^{p}\})<\frac{1}{n},\ \forall n\in\mathbb{N}.$$
 For each $n\in\mathbb{N}$, let $v_{n}=\displaystyle\frac{u_{n}}{\|u_{n}\|}\in F$. Then $\|v_{n}\|=1$ for all $n\in\mathbb{N}$, and
 \begin{equation}\label{measn}
   \meas(\{x\in\mathbb{R}^{N}:\ \xi(x)|v_{n}(x)|^{p}\geq\frac{1}{n}\})<\frac{1}{n},\ \forall n\in\mathbb{N}.
 \end{equation}
 Up to a subsequence, we may assume that $v_{n}\rightarrow v$ in $E$ for some $v\in F$ since $F$ is a finite dimensional space. Clearly $\|v\|=1$. Consequently, there exists a constant $\delta_{0}>0$ such that
 \begin{equation}\label{delta0}
   \meas(\{x\in\mathbb{R}^{N}:\ \xi(x)|v(x)|^{p}\geq\delta_{0}\})\geq \delta_{0}.
 \end{equation}

 In fact, if not, then we have
  $$\meas(\{x\in\mathbb{R}^{N}:\ \xi(x)|v(x)|^{p}\geq\frac{1}{n}\})=0\ \forall n\in\mathbb{N},$$
  which implies that
  $$0\leq \int_{\mathbb{R}^{N}}\xi(x)|v|^{p+r}dx\leq\frac{\|v\|^{r}_{L^{r}(\mathbb{R}^{N})}}{n}\rightarrow0\ \text{as}\ n\rightarrow+\infty.$$
  This together $(f_{1})$ yields $v=0$, which is in contradiction to $\|v\|=1$.

 By using Corollary \ref{cor1} and the fact that all norms are equivalent on $F$, we deduce that

 $$  |v_{n}-v|_{L^{r}(\mathbb{R}^{N})}^{r}\rightarrow0\ \text{as}\ n\rightarrow\infty.$$
 By the H\"{o}lder inequality, it holds that
 \begin{equation}\label{Lr}
   \int_{\mathbb{R}^{N}}\xi(x)|v_{n}-v|^{p}dx\leq \|\xi\|_{\frac{r}{r-p}}\bigg{(}\int_{\mathbb{R}^{N}}|v_{n}-v|^{r}dx\bigg{)}^{\frac{p}{r}}\rightarrow0\ \text{as}\ n\rightarrow\infty.
 \end{equation}
 Set $$\Lambda_{0}:=\{x\in\mathbb{R}^{N}:\ \xi(x)|v(x)|^{p}\geq\delta_{0}\}$$ and for all $n\in\mathbb{N}$,
 $$\Lambda_{n}:=\{x\in\mathbb{R}^{N}:\ \xi(x)|v_{n}(x)|^{p}<\frac{1}{n}\},\ \ \Lambda_{n}^{c}:=\mathbb{R}^{N}\backslash\Lambda_{n}.$$
 Taking into account \eqref{measn} and \eqref{delta0}, we get
 $$\meas(\Lambda_{n}\cap\Lambda_{0})\geq\meas(\Lambda_{0})-\meas(\Lambda_{n}^{c})\geq\delta_{0}-\frac{1}{n}\geq\frac{\delta_{0}}{2},$$
 for $n$ large enough. Therefore we obtain
 \begin{align*}
   \int_{\mathbb{R}^{N}}\xi(x)|v_{n}-v|^{p}dx & \geq \int_{\Lambda_{n}\cap\Lambda_{0}}\xi(x)|v_{n}-v|^{p}dx\\
                                        & \geq \frac{1}{2^{p}}\int_{\Lambda_{n}\cap\Lambda_{0}}\xi(x)|v|^{p}-\int_{\Lambda_{n}\cap\Lambda_{0}}\xi(x)
                                        |v_{n}|^{p})dx\\
                                        & \geq \bigg{(}\frac{\delta_{0}}{2^{p}}-\frac{1}{n}\bigg{)}\meas(\Lambda_{n}\cap\Lambda_{0})\\
                                        & \geq \bigg{(}\frac{\delta_{0}^{2}}{2^{p+2}}\bigg{)}>0
 \end{align*}
 which contradicts \eqref{Lr}. For the $\epsilon$ given in \eqref{epsilon}, let
  $$\Lambda_{u}=\{x\in\mathbb{R}^{N}:\ \xi(x)|u(x)|^{p}\geq\epsilon\|u\|^{p}\},\ \forall u\in F\backslash\{0\}.$$
  Then by \eqref{epsilon},
  $$\meas(\Lambda_{u})\geq\epsilon,\ \forall u\in F\backslash\{0\}.$$
   Therefore
   $$B(u)=\int_{\mathbb{R}^{N}}\xi(x)|u|^{p}dx\geq \int_{\Lambda_{u}}\epsilon\|u\|^{p} dx\geq\epsilon\|u\|^{p}\meas(\Lambda_{u})=\epsilon^{2}\|u\|^{p}. $$
 This implies that $B(u)\rightarrow\infty$ as $\|u\|\rightarrow\infty$ on any finite dimensional subspace of $E$. The proof is complete.
\end{proof}

\begin{lemma}\label{ak}
  Assume that $(m_{1})-(m_{2})$, $(M_{1})-(M_{2})$, $(V_{1})-(V_{2})$ and $(f_{1})$ are satisfied. Then there exists a sequence $\rho_{k}\rightarrow0^{+}$ as $k\rightarrow\infty$ such that $$a_{k}(\lambda)=\inf_{u\in Z_{k},\|u\|=\rho_{k}}I_{\lambda}(u)>0,\ \forall k\in\mathbb{N},$$ and $$d_{k}(\lambda):=\inf_{u\in Z_{k},\|u\|\leq\rho_{k}}I_{\lambda}(u)\rightarrow0\ \text{as}\ k\rightarrow\infty\ \text{uniformly for}\ \lambda\in[1,2].$$
\end{lemma}

\begin{proof}
  Using Lemma \ref{lem3}, for any $u\in Z_{k}$ and $\lambda\in [1,2]$, we can see that

  \begin{align}\label{I}
    I_{\lambda}(u)&\geq \xi_{0}([u]_{(s,M)})+\xi_{0}(\|u\|_{(V,M_{*})})-2\int_{\mathbb{R}^{N}}\xi(x)|u|^{p}dx\\
                  &\geq \xi_{0}([u]_{(s,M)})+\xi_{0}(\|u\|_{(V,M_{*})})-2\|\xi\|_{\frac{r}{r-p}}\|u\|_{L^{r}(\mathbb{R}^{N})}^{p}\nonumber.
  \end{align}
  Let \begin{equation}\label{lk}
        l_{k}=\sup_{u\in Z_{k},\|u\|=1}|u|_{L^{r}(\mathbb{R}^{N})},\ \forall k\in\mathbb{N}.
      \end{equation}

     By the next lemma, there hold
  \begin{equation}\label{lkcv}
    l_{k}\rightarrow0\ \text{as}\ k\rightarrow+\infty.
  \end{equation}

  Combining \eqref{I} and \eqref{lk}, we have
  \begin{equation}\label{I1}
    I_{\lambda}(u)\geq\xi_{0}([u]_{(s,M)})+\xi_{0}(\|u\|_{(V,M_{*})}-2\|\xi\|_{\frac{r}{r-p}}\ l_{k}^{p}\|u\|^{p},\ \forall k\in\mathbb{N}\ \text{and}\ (\lambda,u)\in [1,2]\times Z_{k}.
  \end{equation}
  For each $k\in\mathbb{N}$, choose
  \begin{equation}\label{rhok}
    \rho_{k}=(4 \|\xi\|_{\frac{r}{r-p}}l_{k}^{p})^{\frac{1}{r-p}}.
  \end{equation}
  Since $1<p<r$, then by \eqref{lkcv}, we have
                             \begin{equation}\label{rhok0}
                                \rho_{k}\rightarrow0\ \text{as}\ k\rightarrow+\infty,
                              \end{equation}
    and so, for $k$ large enough, we have $\rho_{k}\leq 1$. Then, by Lemma \ref{lem2},
    \begin{equation}\label{xi}
      \xi_{0}([u]_{(s,M)})+\xi_{0}(\|u\|_{(V,M_{*})}\geq \|u\|^{r},\ \text{when}\ \|u\|=\rho_{k}.
    \end{equation}
    By \eqref{I1}, \eqref{rhok} and \eqref{xi}, direct computation shows
    $$a_{k}(\lambda)=\inf_{u\in Z_{k},\|u\|=\rho_{k}}I_{\lambda}(u)\geq \frac{1}{2}\rho_{k}^{r}>0,\ \forall k\in\mathbb{N}.$$

     Besides, by \eqref{I1}, for each $k\in\mathbb{N}$, we have
     $$I_{\lambda}(u)\geq-2\|\xi\|_{\frac{r}{r-p}}\ l_{k}^{p}\|u\|^{p},$$
      for all $\lambda\in [1,2]$ and $u\in Z_{k}$ with $\|u\|\leq \rho_{k}$. Therefore,
      $$-2\|\xi\|_{\frac{r}{r-p}}\ l_{k}^{p}\|u\|^{p}\leq \inf_{u\in Z_{k},\|u\|\leq\rho_{k}} I_{\lambda}(u)\leq 0,\ \forall\lambda\in [1,2]\ \text{and}\ \forall k\in\mathbb{N}.$$
      Combining \eqref{lkcv} and \eqref{rhok0}, we have
      $$d_{k}(\lambda):=\inf_{u\in Z_{k},\|u\|\leq\rho_{k}}I_{\lambda}(u)\rightarrow0\ \text{as}\ k\rightarrow\infty\ \text{uniformly for}\ \lambda\in[1,2].$$
      The proof is complete.
\end{proof}

\begin{lemma}
We have that
  $$l_{k}:=\sup_{u\in Z_{k},\|u\|=1}|u|_{L^{r}(\mathbb{R}^{N})}\rightarrow0\ \text{as}\ k\rightarrow+\infty.$$
\end{lemma}
\begin{proof}
  It is clear that $l_{k}$ is decreasing with respect to $k$ so there exist $l\geq0$ such that $l_{k}\rightarrow0$ as $k\rightarrow+\infty$. For any $k\geq 0$, there exists $u_{k}\in Z_{k}$ such that $\|u_{k}\|=1$ and $\|u_{k}\|_{L^{r}(\mathbb{R}^{N})}\geq\frac{l_{k}}{2}$. By definition of $Z_{k}$, $u_{k}\rightharpoonup0$ in $E$. Lemma \ref{lem1} implies that $u_{k}\rightarrow0$ in $L^{r}(\mathbb{R}^{N})$. Thus we proved that $l=0$.
\end{proof}

\begin{lemma}\label{bk}
   Under the hypotheses and the sequence $\rho_{k}$ of in Lemma \ref{ak}, there exists $r_{k}>\rho_{k}$ for any $k\in\mathbb{N}$ such that $$b_{k}(\lambda)=\max_{u\in Y_{k},\|u\|=r_{k}}I_{\lambda}(u)<0.$$
\end{lemma}

\begin{proof}
  By using the fact that $Y_{k}$ is with finite dimensional and \eqref{epsilon}, we can find $\epsilon_{k}>0$ such that
  \begin{equation}\label{epsilonk}
  \meas(\Lambda_{u}^{k})\geq\epsilon_{k},\ \forall u\in Y_{k}\setminus{0},
  \end{equation}
  where $\Lambda_{u}^{k}:=\{x\in\mathbb{R}^{N}:\ \xi(x)|u(x)|^{p}\geq\epsilon_{k}\|u\|^{p}\}$. By \eqref{epsilonk}, for any $k\in\mathbb{N}$,
  we have
  \begin{align}\label{Ilambda}
    I_{\lambda}(u) & \leq \xi_{1}([u]_{(s,M)})+\xi_{1}(\|u\|_{(V,M_{*})})-\int_{\mathbb{R}^{N}}\xi(x)|u|^{p}dx\\
                   &\leq \xi_{1}([u]_{(s,M)})+\xi_{1}(\|u\|_{(V,M_{*})})-\int_{\Lambda_{u}^{k}}\epsilon_{k}\|u\|^{p} dx\nonumber\\
                   &\leq \xi_{1}([u]_{(s,M)})+\xi_{1}(\|u\|_{(V,M_{*})})-\epsilon_{k}\|u\|^{p}.\meas(\Lambda_{u}^{k})\nonumber\\
                   &\leq \xi_{1}([u]_{(s,M)})+\xi_{1}(\|u\|_{(V,M_{*})})-\epsilon_{k}^{2}\|u\|^{p}\nonumber\\
                   &\leq \|u\|^{l}-\epsilon_{k}^{2}\|u\|^{p}\leq -\|u\|^{l}\nonumber
  \end{align}
  for all $u\in Y_{k}$ with $\|u\|\leq\min\{\rho_{k},2^{-\frac{1}{l-p}}\epsilon_{k}^{\frac{2}{l-p}},1\}$.
  If we choose $$r_{k}<\min\{\rho_{k},2^{-\frac{1}{l-p}}\epsilon_{k}^{\frac{2}{l-p}},1\},\ \forall k\in\mathbb{N},$$
  and using \eqref{Ilambda}, we deduce that $$b_{k}(\lambda)=\max_{u\in Y_{k},\|u\|=r_{k}}I_{\lambda}(u)=-r_{k}^{l}<0,\ \forall k\in\mathbb{N}.$$

\end{proof}

\begin{proof}[{\bf Proof of Theorem \ref{thm1}}:]
From Lemma \ref{lem1} we can see that $I_{\lambda}$ maps bounded sets to bounded sets uniformly for $\lambda\in[1,2]$. Moreover, $I_{\lambda}$ is even. Then the condition $(i)$ in Theorem \ref{Fountain} is satisfied. Besides, Lemma \ref{Bpositive} shows that the condition $(ii)$ holds while Lemma \ref{ak} together with Lemma \ref{bk} implies that the condition $(iii)$ holds. Therefore, by Theorem \ref{Fountain}, for each $k\in\mathbb{N}$, there exist $\lambda_{n}\rightarrow1,$ $u_{\lambda_{n}}\in Y_{n}$ such that
\begin{equation}\label{lambdan}
 I_{\lambda_{n}}^{'}|_{Y_{n}}(u_{\lambda_{n}})=0,\ I_{\lambda_{n}}(u_{\lambda_{n}})\rightarrow c_{k}\in [d_{k}(2),b_{k}(1)]\ \text{as}\ n\rightarrow\infty.
\end{equation}

\textbf{Claim 2:} The sequence $(u_{\lambda_{n}})_{n\in\mathbb{N}}$ obtained in \eqref{lambdan} is bounded in $E$.\\

 For the sake of notational simplicity, in what follows we always set $u_{n}=u_{\lambda_{n}}$ for all $n\in\mathbb{N}$.\\
In fact, combining \eqref{lambdan}, Lemmas \ref{lem3} and \ref{lem1} and the H\"{o}lder inequality, we obtain
\begin{align*}
  \|u_{n}\|^{\theta} & \leq I_{\lambda_{n}}(u_{n})+\lambda_{n}\int_{\mathbb{R}^{N}}\xi(x)|u_{n}(x)|^{p}dx\leq C_{0}+2\|\xi\|_{\frac{r}{r-p}}\|u_{n}\|_{L^{r}(\mathbb{R}^{N})}^{p}\\
                     & \leq C_{0} + 2\tau^{p}\|\xi\|_{\frac{r}{r-p}}\|u_{n}\|^{p}
\end{align*}
for some $C_{0}>0$ and $\tau>0$, where $\theta=r$ or $\theta=l$.
 Therefore, the claim above is true since $p<\theta$.\\

 \textbf{Claim 3:} The sequence $(u_{n})$ has a strong convergent subsequence for every $k$. \\

In view of Claim $2$, without loss of generality, we may assume \begin{equation}\label{unu0}
                                                                  u_{n}\rightharpoonup u_{0}\ \text{as}\ n\rightarrow+\infty
                                                                \end{equation}
                                                                for some $u_{0}\in E$.
Hence \begin{align*}
	\langle I_{\lambda_{n}}^{'}(u_{n})- I_{\lambda_{n}}^{'}(u_{0}),u_{n}-u_{0}\rangle&=\langle G^{'}(u_{n})-G^{'}(u_{0}),u_{n}-u_{0}\rangle+\int_{\mathbb{R}^{N}}V(x)[m(u_{n})u_{n}-m(u_{0})u_{0}](u_{n}-u_{0})dx\\
&-\lambda_{n}\int_{\mathbb{R}^{N}}[f(x,u_{n})-f(x,u_{0})](u_{n}-u_{0})dx\rightarrow0,\ n\rightarrow+\infty.
	\end{align*}

Therefore, by Lemma \ref{lem4}, we can deduce that $u_{n}$ converges strongly to $u_{0}$ in $E$.\\

Now from the last assertion of Theorem \ref{Fountain}, we know that $I_{1}$ has infinitely many nontrivial critical points. Therefore,
\eqref{eq} possesses infinitely many nontrivial solutions. The proof of Theorem \ref{thm1} is complete.
\end{proof}

\noindent \textbf{\textsc{Sabri Bahrouni}} and \textbf{\textsc{Hichem Ounaies}}\\
Mathematics Department, Faculty of Sciences, University of Monastir,\\
5019 Monastir, Tunisia
 (sabri.bahrouni@fsm.rnu.tn); (hichem.ounaies@fsm.rnu.tn)

\end{document}